\documentclass[12pt,leqno]{article}
\usepackage{amsmath,amssymb,amscd,latexsym,paralist,bm}
\usepackage[all]{xy}

\usepackage[all]{xy}

\newcommand{\C}{{\mathbb{C}}}

\newcommand{\Q}{{\mathbb{Q}}}

\newcommand{\R}{{\mathbb{R}}}
\newcommand{\Z}{{\mathbb{Z}}}

\newcommand{\cont}{\mathrm{cont}}
\newcommand{\ddet}{\mathrm{det}}

\newcommand{\ev}{\mathrm{ev}}
\newcommand{\of}{\overline{f}}

\newcommand{\id}{\mathrm{id}}
\newcommand{\End}{\mathrm{End}\,}
\newcommand{\Hom}{\mathrm{Hom}}

\newcommand{\Imm}{\mathrm{Im}\,}

\newcommand{\Ker}{\mathrm{Ker}\,}

\newcommand{\hGamma}{\hat{\Gamma}}

\newcommand{\Bh}{\mathcal{B}}

\newcommand{\Fh}{{\mathcal F}}

\newcommand{\Nh}{{\cal N}}

\newcommand{\Rh}{{\mathcal R}}

\newcommand{\Wh}{\mathcal{W}}

\newcommand{\hA}{\hat{A}}

\newcommand{\tf}{\tilde{f}}

\newcommand{\iso}{\stackrel{\sim}{\longrightarrow}}
\newcommand{\tei}{\, | \,}

\newcommand{\tgamma}{\tilde{\gamma}}
\newcommand{\tGamma}{\tilde{\Gamma}}
\newcommand{\tS}{\tilde{S}}
\newcommand{\tWh}{\tilde{\Wh}}
\newcommand{\verk}{\mbox{\scriptsize $\,\circ\,$}}
\newcommand{\btu}{\bigtriangleup}

\newtheorem{theorem}{Theorem}

\newtheorem{lemma}[theorem]{Lemma}
\newtheorem{prop}[theorem]{Proposition}
\newtheorem{cor}[theorem]{Corollary}
\newtheorem{example}[theorem]{Example}

\newtheorem{remarks}[theorem]{Remarks}

\newtheorem{question}[theorem]{Question}

\newenvironment{rem}{\noindent {\bf Remark}}{}
\newenvironment{rems}{\noindent {\bf Remarks}}{}

\newenvironment{proofof}{\noindent {\bf Proof of}}{\mbox{}\hfill$\Box$}
\newenvironment{proof}{\noindent {\bf Proof}}{\mbox{}\hspace*{\fill} $\Box$}
\begin{document}
\title{Mahler measures and Fuglede--Kadison determinants}
\author{Christopher Deninger}
\date{\ }
\maketitle
\thispagestyle{empty}

\section{Introduction}
\label{sec:1}

For an essentially bounded complex valued measurable function $P$ on the real $d$-torus $T^d = S^1 \times \ldots \times S^1$ the Mahler measure is defined by the formula $M (P) = \exp m (P) \ge 0$ where $m (P)$ is the integral
\[
m (P) = \int_{T^d} \log |P| \, d\mu \quad \mbox{in} \; \R \cup \{ - \infty \} \; .
\]
Here $\mu$ is the Haar probability measure on $T^d$. If $P$ is a Laurent polynomial on $T^d$ for example, it is known that $\log |P|$ is integrable on $T^d$ unless $P = 0$, so that we have $M (P) > 0$ for $P \neq 0$ and $M (P) = 0$ for $P = 0$.

The Mahler measure appears in many branches of mathematics. It is especially interesting for polynomials with coefficients in $\Z$. If $\alpha$ is an algebraic integer with monic minimal polynomial $P$ over $\Q$ then $m (P)$ is the normalized Weil height of $\alpha$. This follows from an application of Jensen's formula. For polynomials in severable variables there is no closed formula evaluating $m (P)$ but sometimes $m (P)$ can be expressed in terms of special values of $L$-functions, see e.g. \cite{B2}, \cite{L} and their references.

The logarithmic Mahler measure $m (P)$ of a Laurent polynomial $P$ over $\Z$ also appears in ergodic theory as the entropy of a certain subshift defined by $P$ of the full shift for $\Z^d$ with values in the circle c.f. \cite{LSW} and \cite{S}. For relations of $m (P)$ with hyperbolic volumes we refer to \cite{B3}.

We now turn our attention to the determinants in the title. 

Let $\Nh$ be a finite von~Neumann algebra with a faithful normal finite trace $\tau$. In this note we only need the von~Neumann algebra $\Nh\Gamma$ of a discrete group $\Gamma$ which is easy to define, c.f. section \ref{sec:2}. For an invertible operator $A$ in $\Nh$ the Fuglede--Kadison determinant \cite{FK} is defined by the formula
\[
\ddet_{\Nh} A = \exp \tau (\log |A|) \; .
\]
Here $|A| = (A^* A)^{1/2}$ and $\log|A|$ are operators in $\Nh$ obtained by the functional calculus. For arbitrary operators $A$ in $\Nh$ one sets
\[
\ddet_{\Nh} A = \lim_{\varepsilon \to 0+} \ddet_{\Nh} (|A| + \varepsilon) \; .
\]
The main result in \cite{FK} asserts that $\det_{\Nh}$ is {\it multiplicative} on $\Nh$. 
This determinant has several interesting applications. It appears in the definitions of analytic and combinatorial $L^2$-torsion of Laplacians on covering spaces \cite{Lue2}, chapter 3. It was used in the work \cite{HS} on the invariant subspace problem in $\mathrm{II}_1$-factors and it is related to the entropy of algebraic actions of discrete amenable groups \cite{D1}, \cite{DS} and to Ljapunov exponents \cite{D2}.

It was observed in \cite{Lue2}, Example 3.13 that the Mahler measure has the following functional analytic interpretation. For the group $\Gamma = \Z^d$ there is a canonical isomorphism of $\Nh\Gamma$ with $L^{\infty} (T^d , \mu)$ which we write as $A \mapsto \hat{A}$. The relation of Mahler measures with Fuglede--Kadison determinants is then given by the formula:
\begin{equation} \label{eq:1}
\ddet_{\Nh\Z^d} (A) = M (\hat{A}) \quad \mbox{for all} \; A \in \Nh\Z^d \; .
\end{equation}

In this modest note we review certain classical properties of Mahler measures and discuss their generalizations to Fuglede--Kadison determinants of group von~Neumann algebras. In particular, this concerns approximation formulas e.g.~by finite dimensional determinants. Usually the results for Mahler measures are stronger than the corresponding ones for general Fuglede--Kadison determinants and this raises interesting questions. In section \ref{sec:2} we also extend part of the formalism of the theory of orthogonal polynomials on the unit circle to a non-commutative context. Moreover, in section \ref{sec:3} we show that in a suitable sense $\det_{\Nh\Gamma}$ is continuous on the space of marked groups if the argument is invertible in $L^1$. 

\section{Approximation by finite dimensional determinants}
\label{sec:2}

In this section we discuss one way to approximate Mahler measures and more generally Fuglede--Kadison determinants of amenable groups by finite dimensional determinants. Another method which works for residually finite groups is explained in the next section as a special case of theorem \ref{t17}.

The Mahler measure aspect of this topic begins with Szeg\"o's paper \cite{Sz}. For an integrable function $P$ on $S^1$ consider the Fourier coefficients
\[
c_{\nu} = \int_{S^1} z^{-\nu} P (z) \; d\mu (z) \quad \mbox{for} \; \nu \in \Z
\]
and define the following determinants for $n \ge 0$
\[
D_n = \det \left( \vcenter{\xymatrix@=1pt{c_0 & c_1 \ar@{.}[r] & c_{n-1} \\
c_{-1} \ar@{.}[ddd] & c_0 \ar@{.}[ddd]\ar@{.}[r] \ar@{.}[dddr] & c_{n-2} \ar@{.}[ddd] \\
 &  & \\
 & & \\
c_{-n+1} & c_{-n+2} \ar@{.}[r] & c_0
}} \right) \; .
\]
If $P$ is real valued we have $\overline{c}_{\nu} = c_{-\nu}$ and if $P (z) \ge 0$ for all $z \in S^1$ we may view $c_{-\nu}$ as the $\nu$-th moment of the measure $P (z) \, d\mu (z)$. In that case the $D_n$'s are the associated Toeplitz determinants.

\begin{theorem}[Szeg\"o]
  \label{t1}
If $P$ is a continuous real valued function on $S^1$ with $P (z) > 0$ for all $z \in S^1$, then $D_n > 0$ for all $n \ge 0$ and we have the limit formula:
\[
M (P) = \lim_{n \to \infty} \sqrt[n]{D_n} \; .
\]
\end{theorem}

Using the theory of orthogonal polynomials on the unit circle the conditions in Szeg\"o's original theorem have been significantly relaxed, see \cite{Si} for the history:

\begin{theorem}
\label{t2}
  The assertions in Szeg\"o's theorem hold for every real-valued non-negative essentially bounded measurable function $P$ on $S^1$ which is non-zero on a set of positive measure.
\end{theorem}

\begin{proof}
  The $D_n$'s are determinants of Toeplitz matrices for the non-trivial measure $P \, d\mu$. These matrices are positive definit and in particular $D_n > 0$ for every $n \ge 0$, c.f. \cite{Si} section 1.3.2. The limit formula $M (P) = \lim_{n \to \infty} \sqrt[n]{D_n}$ is a special case of \cite{Si} theorem 2.7.14, equality of (i) with (vi) applied to the probability measure $P \| P\|^{-1}_1 \, d\mu$ on $S^1$. (In following that proof, the shortcut in the remark on p.~139 of loc.~cit. is useful.)
\end{proof}

Let us now explain the von~Neumann aspect of these results. For a discrete group $\Gamma$ we will view the elements of $L^p (\Gamma)$ as formal series $\sum_{\gamma \in \Gamma} x_{\gamma} \gamma$ with $\sum_{\gamma} |x_{\gamma}|^p < \infty$. It is then clear that $\Gamma$ acts isometrically by left and right multiplication on $L^p (\Gamma)$. The von~Neumann algebra $\Nh\Gamma$ of $\Gamma$ may be defined as the algebra of bounded operators $A : L^2 \Gamma \to L^2 \Gamma$ which are left $\Gamma$-equivariant. For $\gamma \in \Gamma$ define the unitary operator $R_{\gamma} : L^2 \Gamma \to L^2 \Gamma$ by $R\gamma (x) = x\gamma$. The $\C$-algebra homorphism
\[
r : \C\Gamma \to \Nh\Gamma \quad \mbox{with} \quad r \Big( \textstyle{\sum\limits_{\gamma}} f_{\gamma} \gamma \Big) = \textstyle{\sum\limits_{\gamma}} f_{\gamma} R_{\gamma^{-1}}
\]
extends to a homomorphism $r : L^1 (\Gamma) \to \Nh\Gamma$ with $\| r (f) \| \le \| f \|_1$ for all $f \in L^1 \Gamma$. By looking at $r (f) (e)$ where $e \in \Gamma \subset L^2 \Gamma$ is the unit element of $\Gamma$, we see that $r$ is injective. It will often be viewed as an inclusion in the following. Setting $f^* = \sum \of_{\gamma} \gamma^{-1}$ for $f = \sum f_{\gamma} \gamma$ in $L^1 \Gamma$ the equality $r (f^*) = r (f)^*$ holds. The canonical trace $\tau = \tau_{\Nh\Gamma}$ on $\Nh\Gamma$ is defined by the formula $\tau (A) = (Ae , e)$. It vanishes on commutators $[A,B] = AB - BA$ for $A,B$ in $\Nh\Gamma$. For $f$ in $L^1 \Gamma$ we have $\tau (r(f)) = f_e$. Finally, $\det_{\Nh\Gamma} A$ is defined as in the introduction for every $A$ in $\Nh\Gamma$.

For an abelian group $\Gamma$ with (compact) Pontrjagin dual $\hGamma = \Hom_{\cont} (\Gamma,S^1)$ and Haar probability measure $\mu$ on $\hat{\Gamma}$, the Fourier transform provides an isometry of Hilbert spaces
\[
\Fh : L^2 \Gamma \iso L^2 (\hGamma , \mu) \; .
\]
On the dense subspace $\C\Gamma$ it is given by $\Fh (f) (\chi) = \sum_{\gamma} f_{\gamma} \chi (\gamma)$ for $\chi \in \hGamma$. One can show that under the induced isomorphism of algebras of bounded operators
\[
\Bh (L^2 (\Gamma)) \to \Bh (L^2 (\hGamma , \mu)) \; , \; A \mapsto \Fh \verk A \verk \Fh^{-1}
\]
the von~Neumann algebra $\Nh\Gamma$ maps isomorphically onto $L^{\infty} (\hGamma , \mu)$ where the latter operates by multiplication on $L^2 (\hGamma , \mu)$. Denoting this isomorphism by $A \mapsto \hA$ we have $\hA = \Fh (A(e))$. Namely, $L^2 (\Gamma)$ is a left $\C \Gamma$-algebra and for $f \in \C\Gamma$ we therefore have
\[
\Fh (A (f)) = \Fh (f A (e)) = \Fh (f) \Fh (A (e))\; .
\]
Now the assertion follows because $\Fh (\C\Gamma)$ is dense in $L^2 (\hGamma , \mu)$. It follows that we have
\[
\tau (A) = (Ae,e) = (\Fh (A (e)) , \Fh (e)) = (\hat{A} , 1) = \int_{\hGamma} \hat{A} \, d\mu \; .
\]
Hence there is a commutative diagram
\[
\xymatrix{
\Nh\Gamma \ar[rr]^{\textstyle \hat{\sim}} \ar[dr]_{\textstyle \tau} & & L^{\infty} (\hGamma , \mu) \ar[dl]^{\textstyle \int_{\hGamma}} \\
 & \C &
}
\]
where $\int_{\hGamma}$ denotes integration against the measure $\mu$. We conclude using the definition of $\det_{\Nh\Gamma}$ and Levi's theorem that we have:
\[
\ddet_{\Nh\Gamma} A = \exp \int_{\hGamma} \log |\hA| \, d\mu \quad \mbox{for} \; A \in \Nh\Gamma \; .
\]
In particular, for $\Gamma = \Z^d$ we get formular \eqref{eq:1} from the introduction. It also follows that the generalized Mahler measures studied in \cite{Li2} can be expressed as Fuglede--Kadison determinants.

The non-commutative generalization of Szeg\"o's theorem that we have in mind is valid for amenable groups. A F{\o}lner sequence $(F_n)$ in $\Gamma$ is a sequence of finite subsets $F_n \subset \Gamma$ which are almost invariant in the following sense: For any $\gamma \in \Gamma$ we have
\[
\lim_{n\to\infty} \frac{|F_n \gamma \setminus F_n|}{|F_n|} = 0 \; .
\]
A countable discrete group $\Gamma$ is said to be amenable if it has a F{\o}lner sequence. For example $\Z$ is amenable, the sets  $F_n = \{ 0 , 1 , \ldots , n-1 \}$ forming a F{\o}lner sequence.

For a finite subset $F \subset \Gamma$ and an operator $A \in \Nh\Gamma$ consider the following endomorphism of $\C F$, the finite-dimensional $\C$-vector space over $F$:
\[
A_F : \C F \overset{i_F}{\hookrightarrow} L^2 \Gamma \xrightarrow{A} L^2 \Gamma \xrightarrow{p_F} \C F \; .
\]
Here $i_F$ is the inclusion and $p_F$ the orthogonal projection to $\C F$. We have $p^*_F = i_F$ for the $L^2$-adjoints and hence $(A_F)^* = (A^*)_F$.

\begin{lemma}
  \label{t3}
If $A \in \Nh\Gamma$ is positive then $A_F$ is positive as well and hence $\det A_F \ge 0$. If $A$ is positive, and injective on $\C\Gamma$ then $A_F$ is a positive automorphism of $\C F$ and hence $\det A_F > 0$.
\end{lemma}

\begin{proof}
  Set $B = \sqrt{A}$. For $v \in \C F$ we have $(A_F v , v) = (Av , v) = \| Bv \|^2$ and hence $A_F$ is positive. Moreover $A_F v = 0$ implies $Bv = 0$ and hence $Av = B (Bv) = 0$. If $A$ is injective on $\C \Gamma$ we get $v = 0$ and thus $A_F$ is injective and hence an automorphism.
\end{proof}

The approximation result corresponding to Szeg\"o's theorem is the following one which was proved in \cite{D1} Theorem 3.2:

\begin{theorem}
  \label{t4}
Let $\Gamma$ be a finitely generated amenable group with a F{\o}lner sequence $(F_n)$ and let $A$ be a positive invertible operator in $\Nh\Gamma$. Then $\det A_{F_n} > 0$ for all $n$ and we have:
\[
\ddet_{\Nh\Gamma} A = \lim_{n\to\infty} (\det A_{F_n})^{1/|F_n|}\; .
\]
\end{theorem}

Positivity of $\det A_{F_n}$ follows from lemma \ref{t3}. The proof of the limit formula is based on an approximation result for traces of polynomials in $A$ due to Schick \cite{Sc} generalizing previous work of L\"uck. Theorem \ref{t4} follows by applying the Weierstra{\ss} approximation theorem to $\log$ and the fact that the spectrum of $A$ and all $A_{F_n}$ is uniformly bounded away from zero. 

We would like to point out that another part of Szeg\"o's theory which characterizes $M (P)$ by an extremal property has been generalized to the setting of von~Neumann algebras in \cite{BL}.

\begin{example} \label{t5} \rm
Let us now show that Szeg\"o's theorem \ref{t1} is a special case of theorem \ref{t4}. Consider a measurable essentially bounded function $P : S^1 \to \R$ with $P (z) \ge 0$ for all $z \in S^1$. It defines a positive element $P$ of the von~Neumann algebra $L^{\infty} (S^1 , \mu)$. Let $A$ be the positive operator in $\Nh\Z$ with $\hA = P$, i.e. with $\Fh (A (0)) = P$. For $\nu \in \Z \subset L^2 (\Z)$ write $(\nu)$ for its image in $L^2 (\Z)$. Then we have $\Fh (\nu) = z^{\nu}$ viewed as a character on $S^1$. Thus
\[
c_{\nu} = \int_{S^1} z^{-\nu} P (z) \, d\mu (z) = (P, z^{\nu}) = (\Fh (A (0)) , \Fh (\nu)) = (A (0) , (\nu)) \; .
\]
Now consider the F{\o}lner sequence $F_n = \{ 0 , 1 , \ldots , n-1 \}$ of $\Z$. The matrix of $A_{F_n}$ with respect to the basis $(0) , (1) , \ldots , (n-1)$ of $\C F_n$ has $(i,j)$-th coefficient
\[
(A_{F_n} (i) , (j)) = (A (i) , (j)) = (A (0) , (j-i)) = c_{j-i} \; .
\]
Thus we have $\det A_{F_n} = D_n$ and therefore theorem \ref{t4} implies theorem \ref{t1} (even with ``continuous'' replaced by ``measurable essentially bounded''). 
\end{example}
The analogue of theorem \ref{t2} in our setting does not seem to be known. We formulate it as a question:
\begin{question}
  \label{t6}
Let $\Gamma$ be a finitely generated amenable group and $A$ a positive operator in $\Nh\Gamma$. Does the limit formula
\[
\ddet_{\Nh\Gamma} A = \lim_{n\to\infty} (\det A_{F_n})^{1 / |F_n|}
\]
hold for every F{\o}lner sequence?
\end{question}

\begin{rems}
  1) In theorem \ref{t2} the non-zero positive operators in $\Nh\Z \cong L^{\infty} (S^1 , \mu)$ were considered. These are injective on $\C\Z$ because $\Fh (\C\Z) = \C [z,z^{-1}]$, and non-zero Laurent polynomials vanish only in a set of measure zero on $S^1$. Perhaps it is reasonable therefore to first consider only positive operators which are injective on $\C\Gamma$ so that by lemma \ref{t3} all $A_{F_n}$ are positive automorphisms. On the other hand, for $A = 0$ the limit formula is trivially true.\\
2) Because of the next proposition it would suffice to prove the inequality
\[
\ddet_{\Nh\Gamma} A \le \varliminf_{n\to\infty} (\det A_{F_n})^{1 /|F_n|}
\]
in order to answer question \ref{t5} affirmatively.
\end{rems}

\begin{prop}
  \label{t7}
For a finitely generated group $\Gamma$ and any positive operator $A$ on $\Nh\Gamma$ we have
\[
\ddet_{\Nh\Gamma} A \ge \varlimsup_{n\to\infty} (\det A_{F_n})^{1/|F_n|} \; .
\]
\end{prop}

\begin{proof}
  For $A$ in $\Z\Gamma$ this is proved in \cite{Sc}. In general we can argue as follows. For any endomorphism $\varphi$ set $\varphi^{(\varepsilon)} = \varphi + \varepsilon \id$. Then we have $(A^{(\varepsilon)})_F = (A_F)^{(\varepsilon)}$ for finite $F \subset \Gamma$. The following relations hold:
  \begin{eqnarray*}
    \ddet_{\Nh\Gamma} A & \overset{(i)}{=} & \lim_{\varepsilon \to 0+} \ddet_{\Nh\Gamma} A^{(\varepsilon)} \overset{(ii)}{=} \lim_{\varepsilon \to 0+} \lim_{n\to\infty} (\det (A^{(\varepsilon)})_{F_n})^{1/|F_n|} \\
& \overset{(iii)}{\ge} & \varlimsup_{n\to\infty} (\det A_{F_n})^{1 / |F_n|} \; .
  \end{eqnarray*}
Here (i) is true by the definition of the Fuglede--Kadison determinant and (ii) follows from theorem \ref{t4} applied to $A^{(\varepsilon)}$. Finally (iii) holds because $\det (A_{F_n})^{(\varepsilon)} \ge \det A_{F_n}$ for every $n \ge 1$ and $\varepsilon > 0$.
\end{proof}

In the rest of this section we develop a formalism for the determinants $\det A_F$ and $\det_{\Nh\Gamma}A$ which is suggested by the theory of orthogonal polynomials on the unit circle. We also point out the relation to question \ref{t6}.

We start with the following well known lemma:

\begin{lemma}
  \label{t8}
For a block matrix over a field with $A$ invertible the following formula holds:
\[
\det \left(
  \begin{smallmatrix}
    A & B \\ C & D 
  \end{smallmatrix} \right) = \det (D-CA^{-1} B) \det A \; .
\]
\end{lemma}

\begin{proof}
  We have
\[
\det \left(
  \begin{smallmatrix}
    A & B \\ C & D 
  \end{smallmatrix} \right) = \det \left(
  \begin{smallmatrix}
    I & B \\ CA^{-1} & D
  \end{smallmatrix} \right)  \det \left(
  \begin{smallmatrix}
    A & 0 \\ 0 & I
  \end{smallmatrix} \right) = \det \left(
  \begin{smallmatrix}
    I & B \\ 0 & D-CA^{-1} B
  \end{smallmatrix} \right) \det A \; .
\]
\end{proof}

Consider a countable discrete group $\Gamma$ and finite subsets $F \subset F' \subset \Gamma$. Let $A \in \Nh \Gamma$ be positive, and injective on $\C\Gamma$, so that according to lemma \ref{t3} the endomorphism $A_F$ is positive and invertible. In terms of the decomposition 
\[
\C F' = \C F \oplus \C (F' \setminus F)
\]
the endomorphism $A_{F'}$ is given by the block matrix
\[
A_{F'} =
\begin{pmatrix}
  A_F & p_F A i_{F' \setminus F} \\
p_{F' \setminus F} A i_F & p_{F' \setminus F} Ai_{F' \setminus F}
\end{pmatrix} \; .
\]
Thus lemma \ref{t8} gives the formula
\begin{equation}
  \label{eq:2}
  \det A_{F'} = \det A_F \det (p_{F' \setminus F} A i_{F' \setminus F} - p_{F' \setminus F} A i_F A^{-1}_F p_F Ai_{F' \setminus F}) \; .
\end{equation}
Now consider the endomorphism
\[
\psi = i_F A^{-1}_F p_F Ai_{F'} : \C F' \to \C F'
\]
and the scalar product on $\C\Gamma$ defined by
\[
(u,v)_A := (Au,v) = (u,Av) \; .
\]
It is positive since for $u,v \in \C \Gamma$ there is a finite subset $F \subset \Gamma$ with $u,v \in \C F$ and then we have $(u,v)_A = (A_F u,v)$ with the positive automorphism $A_F$.

\begin{prop}
  \label{t9}
The endomorphism $\psi$ is the orthogonal projection of $\C F'$ to $\C F$ with respect to the scalar product $(,)_A$ on $\C F'$.
\end{prop}

\begin{proof}
  For $u \in \C F$ we have $\psi (u) = A^{-1}_F A_F u = u$. This implies that $\psi^2 = \psi$ since $\psi$ takes values in $\C F$. Moreover, $\Imm \psi = \C F$. Next observe that
\[
p_{F'} A \psi = p_{F'} A i_F A^{-1}_F p_F A i_{F'}
\]
is selfadjoint since $i^*_F = p_F$ and $A,A_F$ are selfadjoint. Hence we have $p_{F'} A \psi = \psi^* A i_{F'}$ and for $u,v \in \C F'$ therefore:
\[
(\psi u,v)_A = (\psi u , Av) = (u , \psi^* A v) = (u, A \psi v) = (u, \psi v)_A \; .
\]
\end{proof}

By the proposition the endomorphism $\varphi = \id - \psi$ of $\C F'$ is the orthogonal projection to $\C F \; ^{\perp_A}$ with respect to $(,)_A$. Formula \eqref{eq:2} can be rewritten as
\[
\det A_{F'} = \det A_F \det (p_{F' \setminus F} A \varphi i_{F' \setminus F}) \; .
\]

\begin{cor}
  \label{t10}
Assume that $F' = F \dot{\cup} \{ \gamma \}$ and set $\Phi_{\gamma} = \varphi (\gamma)$. Then we have $\det A_{F'} = \| \Phi_{\gamma} \|^2_A \det A_F$.
\end{cor}

\begin{proof}
  Since $\C (F' \setminus F) = \C \gamma$ is one-dimensional and $\| \gamma \| = 1$, we have
  \begin{eqnarray*}
    \det (p_{F' \setminus F} A \varphi i_{F' \setminus F}) & = & (p_{F' \setminus F} A \varphi (\gamma) , \gamma) = (A \varphi (\gamma) , \gamma) = (\varphi (\gamma) , \gamma)_A \\
 & = & (\varphi^2 (\gamma) , \gamma)_A = (\varphi (\gamma) , \varphi (\gamma))_A = \| \varphi (\gamma) \|^2_A \; .
  \end{eqnarray*}
\end{proof}

The corollary generalizes part of formula (1.5.78) of \cite{Si}. Using this orthogonalization process inductively we get the formula
\[
\det A_F = \prod_{\gamma \in F} \| \Phi_{\gamma} \|^2_A \; .
\]
Concerning the order of $\| \Phi_{\gamma} \|_A$ note the following equations
\begin{eqnarray*}
  \| \Phi_{\gamma} \|^2_A & = & (\varphi (\gamma) , \gamma)_A = (\gamma - i_F A^{-1}_F p_F A \gamma , \gamma)_A \\
 & = & \| \gamma \|^2_A - (i_F A^{-1}_F p_F A \gamma , A\gamma) \\
 & = & \tau (A) - (A^{-1}_F s, s) \; , \; \mbox{where} \; s = p_F A \gamma \in \C F \; .
\end{eqnarray*}
In the situation of corollary \ref{t10} we therefore obtain:

\begin{cor}
  We have $0 < \| \Phi \|^2_A \le \tau (A)$. Moreover the following assertions are equivalent:\\
1) $\| \Phi_{\gamma} \|^2_A = \tau (A)$\\
2) $p_F A\gamma = 0$\\
3) $A_{F'} = \left(
  \begin{smallmatrix}
    A_F &  0 \\ 
0 & c
  \end{smallmatrix} \right)$ for some $c$ (which must be $c = \tau (A)$).
\end{cor}

Now we generalize a calculation from the theory of orthogonal polynomials on $S^1$ which is used in one of the proofs of theorem \ref{t2}. Recall that for $\Phi \in \Nh\Gamma$ one sets $\| \Phi \|_2 = \tau (\Phi^* \Phi)^{1/2}$. It is known that we have
\begin{equation}
  \label{eq:3}
  \ddet_{\Nh\Gamma} \Phi \le \| \Phi \|_2 \; .
\end{equation}
Namely, let $E_{\lambda}$ be the spectral resolution of $|\Phi|$. Then we have by Jensen's inequality:
\begin{eqnarray*}
  (\ddet_{\Nh\Gamma} \Phi)^2 & = & \exp \int^{\infty}_0 \log (|\lambda|^2) d \tau (E_{\lambda}) \le \int^{\infty}_0 |\lambda|^2 \, d\tau (E_{\lambda}) \\
& = & \tau \Big( \int^{\infty}_0 |\lambda |^2 d E_{\lambda} \Big) = \tau (\Phi^* \Phi) = \| \Phi \|^2_2 \; .
\end{eqnarray*}
For positive $A \in \Nh\Gamma$ and any $\Phi \in \Nh\Gamma$ we find
\begin{eqnarray*}
  (\ddet_{\Nh\Gamma} A)^{1/2} \ddet_{\Nh\Gamma} \Phi & = & \ddet_{\Nh\Gamma} (\sqrt{A} \Phi) \le \| \sqrt{A} \Phi \|_2 \\
  & = & \tau (\Phi^* A \Phi)^{1/2} = (\Phi^* A \Phi e , e)^{1/2} \\
 & = & (A \Phi (e) , \Phi (e))^{1/2} = \| \Phi (e) \|_A \; .
\end{eqnarray*}
Let $\sim : \C \Gamma \to \C\Gamma$ be defined by $\tilde{f} = \sum f_{\gamma} \gamma^{-1}$. Then for $f \in \C\Gamma$ the operator $r (f) \in \Nh\Gamma$ is right multiplication by $\tilde{f}$. For $f \in \C \Gamma \subset \Nh\Gamma$ where the inclusion is via $r$, we get 
\[
(\ddet_{\Nh\Gamma} A)^{1/2} \ddet_{\Nh\Gamma} f \le \| f (e) \|_A = \| \tilde{f} \|_A \; .
\]
Applying this to $f = \tilde{\Phi}_{\gamma}$ we find
\[
(\ddet_{\Nh\Gamma} A)^{1/2} \ddet_{\Nh\Gamma} \tilde{\Phi}_{\gamma} \le \| \Phi_{\gamma} \|_A \; .
\]
Combining this with corollary \ref{t10} we obtain the following result:

\begin{cor}
  \label{t12}
Let $A \in \Nh\Gamma$ be positive, and injective on $\C\Gamma$. Assume that $F$ and $F' = F \dot{\cup} \{ \gamma \}$ are finite subsets of $\Gamma$. Then we have the inequality:
\[
(\ddet_{\Nh\Gamma} A) (\ddet_{\Nh\Gamma} \tilde{\Phi}_{\gamma} )^2 \le \frac{\det A_{F'}}{\det A_F} \; .
\]
\end{cor}

\begin{remarks}
  \label{t13}\rm
{\bf a} For $\Gamma \in \Z$ and $F = \{ 0 , \ldots , n-1 \}$ and $F' = \{ 0 , \ldots , n \}$ we have $\gamma = n$ and for $A$ as in example \ref{t5} the Laurent polynomial $\Fh (\tilde{\Phi}_n)$ is closely related to the polynomial $\Phi^*_n (z)$ of \cite{Si} (1.1.1), (1.1.6). In particular, $\Fh (\tilde{\Phi}_n)$ like $\Phi^*_n (z)$ has no zeroes in the interior of the unit disc and hence $\det_{\Nh\Z} \tilde{\Phi}_n = M (\Fh (\tilde{\Phi}_n)) = 1$ by Jensen's formula, c.f. \cite{Si} Theorem 1.7.1 and proof of Corollary 1.7.2. Thus in this case one has
\[
M (\hat{A}) = \ddet_{\Nh\Gamma} A \le \frac{\det A_{F_{n+1}}}{\det A_{F_n}} \; ,
\]
and this inequality is instrumental for the proof of theorem \ref{t2}.\\
{\bf b} For general $\Gamma$ unfortunately we do not know whether $\det_{\Nh\Gamma} \tilde{\Phi}_{\gamma} \ge 1$ or even $\det_{\Nh\Gamma} \tilde{\Phi}_{\gamma} = 1$ holds under suitable conditions.
\end{remarks}
\section{Approximation on the space of marked groups}
\label{sec:3}

According to a theorem of Lawton, Mahler measures of Laurent polynomials in several variables can be approximated by Mahler measures of one-variable Laurent polynomials. His result which we now recall resolved a conjecture of Boyd. For $r \in \Z^d$ set 
\[
q (r) = \min \{ \| \nu \| \tei 0 \neq \nu \in \Z^d \; \mbox{with} \; (\nu , r) = 0 \}
\]
where $\| \nu \| = \max |\nu_i| $ and $(\nu , r) = \sum_i \nu_i r_i$.

\begin{theorem}[Lawton \cite{La}]
  \label{t14}
For\ \;$r \in \Z^d$ and $P$ in $\C [X^{\pm 1}_1 , \ldots , X^{\pm 1}_d ]$ set $P_r (X) = P (X^{r_1} , \ldots , X^{r_d})$. Then we have
\[
\lim_{q (r) \to \infty} M (P_r) = M (P) \; .
\]
\end{theorem}

If $P$ does not vanish on $T^d$, so that $\log |P_r|$ is continuous on $S^1$ the theorem is much simpler to prove than in general. In the following we will generalize this easy case to a statement on the continuity of the Fuglede--Kadison determinant on the space of marked groups. A full generalization of theorem \ref{t14} in this direction is a challenging problem.

For $d \ge 1$ the space $X_d$ of marked groups on $d$-generators is the set of isomorphism classes $[\Gamma , S]$ of pairs $(\Gamma , S)$ where $\Gamma$ is a discrete group and $S = (s_1 , \ldots , s_d)$ a family of $d$ generators of $\Gamma$. Here repetitions are allowed. Two such pairs $(\Gamma,S)$ and $(\Gamma' , S')$ are called isomorphic if there is an isomorphism $\alpha : \Gamma \iso \Gamma'$ with $\alpha (S) = S'$. The set $X_d$ becomes an ultra-metric space with the distance function
\[
d ([\Gamma_1 , S_1] , [\Gamma_2 , S_2]) = 2^{-N}
\]
where $N \le \infty$ is the largest radius such that the balls of radius $N$ around the origin in the Cayley graphs of $(\Gamma_1 , S_1)$ and $(\Gamma_2 , S_2)$ are isomorphic as oriented, labelled graphs with labels $1 , \ldots , d$ corresponding to the generators. Thus intuitively two marked groups are close to each other if their Cayley graphs around the origin coincide on a big ball. An equivalent metric on $X_d$ is obtained by setting
\[
\delta ([\Gamma_1 , S_1] , [\Gamma_2 , S_2]) = 2^{-M}
\]
if the bijection $S_1 \cong S_2$ induces a bijection of $S_1$- resp.~$S_2$-relations of length less than $M$ and if $M \le \infty$ is maximal with this property. Here an $S$-relation in a group $\Gamma$ is an $S$-word, i.e. a finite string of elements from $S$ and their inverses, whose evalution in $\Gamma$ is equal to $e$. The number of elements in the string defining a word is the length of the word e.g. $s^{-1}_1 s_2 s_1 s^{-1}_3 s_5$ has length $5$.

Much more background on the space of marked groups can be found in \cite{CG} section 2, for example.

\begin{example}
  \label{t15} \rm
With notations as in theorem \ref{t14} consider
\[
(\Gamma , S) = (\Z^d , e_1 , \ldots , e_d) \quad \mbox{and} \quad (\Gamma_r , S_r) = (D (r) \Z , r_1 , \ldots , r_d) 
\]
where $r \in \Z^d$ and $D (r)$ is the greatest common divisor of $r_1 , \ldots , r_d$. Then we have
\[
\lim_{q (r) \to \infty} [\Gamma_r , S_r] = [\Gamma , S] \quad \mbox{in} \; X_d \; .
\]
\end{example}

\begin{proof}
  As $\Gamma_r$ is abelian, an $S_r$-word is a relation in $\Gamma_r$ if and only if $\sum^d_{i=1} \nu_i r_i = 0$ where $\nu_i \in \Z$ is the sum of all exponents $\pm 1$ of $r_i$ in the word. The length of the relation is at least $\| \nu \|$. If a relation $\Rh$ has length less than $q (r)$ it follows that we have $\nu = 0$ and hence $\Rh$ is a relation of commutation. Hence for length less than $q (r)$ the relations in $(\Gamma , S)$ and $(\Gamma_r , S_r)$ are in canonical bijection. Thus we have
\[
\delta ([\Gamma , S] , [\Gamma_r , S_r]) \le 2^{-q (r)}
\]
and the assertion follows.
\end{proof}

\begin{example}
  \label{t16} \rm
Let $\Gamma$ be a countable group and $(K_n)$ a sequence of normal subgroups of $\Gamma$. We write $K_n\to e$ if $e$ is the only element of $\Gamma$ which is contained in infinitely many $K_n$'s. Equivalently, for any finite subset $Q \subset \Gamma$ we have $K_n \cap Q \subset \{ e \}$ for $n$ large enough. \\
Now assume that $\Gamma$ is finitely generated and let $S$ be a finite familiy of generators. Given epimorphisms $\varphi_n : \Gamma \twoheadrightarrow \Gamma_n$ we get finite families of generators $S_n$ in $\Gamma_n$. Setting $d = |S|$, we claim that the limit formula
\[
\lim_{n \to \infty} [\Gamma_n , S_n] = [\Gamma,S] \quad \mbox{in} \; X_d 
\]
is equivalent to $K_n \to e$, where $K_n = \Ker \varphi_n$.
\end{example}

\begin{proof}
  Assume that $K_n \to e$. Let $\Rh_n$ be a relation of length $l$ in $\Gamma_n$ and let $\Rh$ be the corresponding $S$-word in $\Gamma$. The evaluation $\gamma = \ev (\Rh)$ of $\Rh$ in $\Gamma$ lies in $K_n$. Let $Q \subset \Gamma$ be the finite subset of at most $l$-fold products from $S \cup S^{-1}$. In particular $\gamma \in Q$. For $n \ge n (l)$, we have $K_n \cap Q \subset \{ e \}$ since $K_n \to e$. Therefore the relations of length $\le l$ in $\Gamma$ and $\Gamma_n$ are in canonical bijection if $n \ge n (l)$ and hence we have
\[
\delta ([\Gamma_n , S_n] , [\Gamma,S]) \le 2^{-l} \quad \mbox{for} \; n \ge n (l) \; .
\]
For the converse consider an element $\gamma \in \Gamma$ which is contained in infinitely many $K_n$'s. Choose a word $\Wh$ in $\Gamma$ with $\gamma = \ev (\Wh)$ and let $l$ be the length of $\Wh$. By assumption, there are arbitrarily large $n$'s such that $\varphi_n (\Wh)$ is a relation in $\Gamma_n$. But for $n \gg 0$ the relations of length $l$ in $\Gamma_n$ and $\Gamma$ are in bijection. Hence $\Wh$ must be a relation i.e. $\gamma = \ev (\Wh) = e$. 
\end{proof}

In order to state the next result we introduce some notations.

For a homomorphism $\varphi : \Gamma \to\Gamma'$ of discrete groups denote by $\varphi_* : L^1 (\Gamma) \to L^1 (\Gamma')$ the map ``integration along the fibres'' defined by
\[
\varphi_* \Big( \textstyle \sum\limits_{\gamma\in\Gamma} f_{\gamma} \gamma \Big) = \sum\limits_{\gamma \in \Gamma} f_{\gamma} \varphi (\gamma) = \sum\limits_{\gamma' \in \Gamma'} \Big( \sum\limits_{\gamma \in \varphi^{-1} (\gamma')} f_{\gamma} \Big) \gamma' \; .
\]
The map $\varphi_*$ is a homomorphism of Banach $*$-algebras with units and it satisfies the estimate $\| \varphi_* (f) \|_1 \le \| f \|_1$ for all $f \in L^1 (\Gamma)$. 

Recall that we view $L^1 (\Gamma)$ as a subalgebra of $\Nh\Gamma$. Then we have the following result

\begin{theorem}
  \label{t17}
Consider a countable discrete group together with homomorphisms $\varphi_n : \Gamma \to \Gamma_n$. For $f \in L^1 (\Gamma)$ set $f_n = \varphi_{n*} (f) \in L^1 (\Gamma_n)$. Then we have:
\begin{equation} \label{eq:3a}
\ddet_{\Nh\Gamma} f \ge \varlimsup_{n\to\infty} \ddet_{\Nh\Gamma_n} f_n \quad \mbox{if} \; K_n = \Ker \varphi_n \to e \; .
\end{equation}
In case $f \in L^1 (\Gamma)^{\times}$, equality holds:
\begin{equation} \label{eq:3b}
\ddet_{\Nh\Gamma} f = \lim_{n\to\infty} \ddet_{\Nh\Gamma_n} f_n \quad \mbox{if} \; K_n \to e \; .
\end{equation}
\end{theorem}

In particular, using example \ref{t16} we get the following corollary:

\begin{cor}
  \label{t18}
For $[\Gamma , S] \in X_d$, epimorphisms $\varphi_n : \Gamma \twoheadrightarrow \Gamma_n$ and $f \in L^1 (\Gamma)^{\times}$ we have
\[
\ddet_{\Nh\Gamma} f = \lim_{n\to \infty} \ddet_{\Nh\Gamma_n} f_n \quad \mbox{if} \;\; [\Gamma_n , S_n] \to [\Gamma , S] \; \mbox{in} \; X_d \; .
\]
\end{cor}

Let us give two examples:

\begin{example}
  \label{t19} \rm
For any countable residually finite group $\Gamma$ there is a sequence of normal subgroups $K_n$ with finite index such that $K_n \to e$. Set $\Gamma_n = K_n \setminus \Gamma$. Then we have:
\begin{equation}
  \label{eq:4}
  \ddet_{\Nh\Gamma} f = \lim_{n\to \infty} |\det r (f_n)|^{1/|\Gamma_n|} \quad \mbox{for any} \; f \in L^1 (\Gamma)^{\times} \; .
\end{equation}
Note here that $r (f_n) \in \Nh\Gamma_n \subset \End \C \Gamma_n$. This formula follows immediately from theorem \ref{t17} if we note that for a finite group $G$ and an element $h \in \C G = L^1 (G)$ we have:
\[
\ddet_{\Nh G} h = |\det r (h)|^{1/ |G|} \; .
\]
Formula \eqref{eq:4} was used in \cite{DS} to relate the growth rate of periodic points of certain algebraic $\Gamma$-actions to Fuglede--Kadison determinants. For $f$ in $\Z\Gamma \cap L^1 (\Gamma)^{\times}$ formula \eqref{eq:4} is a special case of \cite{Lue1}, Theorem 3.4,3.
\end{example}

\begin{example}
  \label{t20}
\rm Recall the situation of example \ref{t15} and let $P$ be a continuous function on $T^d$ whose Fourier coefficients are absolutely summable. Thus we have $P = \Fh (f)$ for some $f \in L^1 (\Z^d)$. If we assume that $P$ does not vanish in any point of $T^d$ it follows from a theorem of Wiener \cite{Wi} that we have $f \in L^1 (\Z^d)^{\times}$. Define $\varphi_r : \Gamma = \Z^d \to \Gamma_r$ by $\varphi_r (e_i) = r_i$ for $1 \le i \le r$. Corollary \ref{t18} now implies the formula
\[
\ddet_{\Nh\Gamma} f = \lim_{q (r) \to\infty} \ddet_{\Nh\Gamma_r} f_r \; .
\]
Since $\det_{\Nh\Gamma} f = M (P)$ and $\ddet_{\Nh\Gamma_r} f_r = M (P_r)$, we get the limit formula of Lawton's theorem in this (easy) case. 
\end{example}

The theorem of Wiener mentioned above has been generalized to the non-commutative context. The ultimate result is due to Losert \cite{Lo}. It asserts that $L^1 (\Gamma)^{\times} = L^1 (\Gamma) \cap C^* (\Gamma)^{\times}$ if and only if $\Gamma$ is ``symmetric''. Thus for symmetric groups the question of invertibility in $L^1 (\Gamma)$ is reduced to the easier question of invertibility in the $C^*$-algebra $C^* (\Gamma)$. Finitely generated virtually nilpotent discrete groups for example are known to be symmetric \cite{LP}, Corollary 3, and hence we have the following equalities for them
\begin{equation}
  \label{eq:5}
  L^1 (\Gamma)^{\times} = L^1 (\Gamma) \cap C^*_r (\Gamma)^{\times} = L^1 (\Gamma) \cap (\Nh\Gamma)^{\times} \; .
\end{equation}

Note that for amenable groups the $C^*$-algebra and the reduced $C^*$-algebra coincide. The classical Wiener theorem is a special case of \eqref{eq:5}:
\[
L^1 (\Z^d)^{\times} = L^1 (\Z^d) \cap C^0 (T^d)^{\times} = L^1 (\Z^d) \cap L^{\infty} (T^d , \mu)^{\times} \; .
\]
The assumptions in corollary \ref{t18} are more restrictive than in theorem \ref{t17}. The advantage of its formulation lies in the intuition and results about $X_d$ that one may use.

\begin{proofof}
  {\bf theorem \ref{t17}}
First we need a simple result about traces. We claim that for any $f$ in $L^1 (\Gamma)$ and any complex polynomial $P (X)$ we have
\begin{equation}
  \label{eq:6}
  \tau_{\Nh\Gamma} (P (f)) = \lim_{n\to\infty} \tau_{\Nh\Gamma_n} (P (f_n)) \quad \mbox{if} \; K_n \to e \; .
\end{equation}
Since $P (f)$ lies in $L^1 (\Gamma)$ as well, it suffices to prove \eqref{eq:6} for $P (X) = X$. Writing $f = \sum_{\gamma} f_{\gamma} \gamma$ we have $\tau_{\Nh\Gamma} (f) = f_e$ and $\tau_{\Nh\Gamma_n} (f_n) = \sum_{\gamma \in K_n} f_{\gamma}$. Fix $\varepsilon > 0$. Since $f$ is in $L^1 (\Gamma)$, there is a finite subset $Q \subset \Gamma$ with $\sum_{\gamma \in \Gamma \setminus Q} |f_{\gamma}| < \varepsilon$. Because of the assumption $K_n \to e$, there is some $N \ge 1$ such that $K_n \cap Q \subset \{ e \}$ for all $n \ge N$. For $n \ge N$ we therefore get the estimate
\[
|\tau_{\Nh\Gamma} (f) - \tau_{\Nh\Gamma_n} (f_n)| = |f_e - \textstyle \sum\limits_{\gamma \in K_n} f_{\gamma} | \le \sum\limits_{\gamma \in K_n \setminus e} |f_{\gamma}| \le \sum\limits_{\gamma \in  \Gamma\setminus Q} |f_{\gamma}| < \varepsilon \; .
\]
Since $\varepsilon > 0$ was arbitrary, formula \eqref{eq:6} follows.

Next, for any $f \in L^1 (\Gamma)$ we have
\begin{equation}
  \label{eq:7}
  \| r (f_n)\| \le \| f_n \|_1 \le \|f \|_1 \quad \mbox{and} \quad \| r (f) \| \le \| f \|_1
\end{equation}
where $\|\;\|$ is the operator norm (between $L^2$-spaces).

Moreover, if $f \in L^1 (\Gamma)^{\times}$, the relation $(f^{-1})_n = \varphi_n (f^{-1}) = \varphi_n (f)^{-1} = f^{-1}_n$ implies the estimates:
\begin{equation}
  \label{eq:8}
  \| r (f^{-1}_n) \| \le \| f^{-1}_n \|_1 \le \| f^{-1} \|_1 \; .
\end{equation}
Since $2  \det_{\Nh\Gamma} f = \det_{\Nh\Gamma} f^* f$ and
\[
(f^* f)_n = \varphi_n (f^* f) = \varphi_n (f)^* \varphi_n (f) = f^*_n f_n 
\]
we may replace $f$ by $f^* f$ in the assertion of theorem \ref{t17}. Hence we may assume that $f \in L^1 (\Gamma)$ and $f_n \in L^1 (\Gamma_n)$ are positive in $\Nh\Gamma$ resp. $\Nh\Gamma_n$ i.e. that $r (f)$ and $r  (f_n)$ are positive operators. If $f$ is invertible it follows that the spectrum of $r (f)$ is contained in the interval $I = [ \| f^{-1} \|^{-1}_1 , \| f \|_1 ]$. Using the estimates \eqref{eq:7} and \eqref{eq:8} we see that the spectra of $r (f_n)$ lie in $I$ as well for all $n$. Note here that for a positive bounded operator $A$ on a Hilbert space we have $\| A \| = \max_{\lambda \in \sigma (A)} \lambda$. Fix $\varepsilon > 0$. Since $I$ is a compact subinterval of $(0,\infty)$, it follows from the Weierstra{\ss} approximation theorem that there is a polynomial $P (X)$ with $\max_{x\in I} |P (x) - \log x| \le \varepsilon$. Since $\sigma (r (f)) , \sigma (r (f_n))$ lie in $I$ it follows that we have:
\[
\| \log r(f) - P (r (f)) \| \le \varepsilon \quad \mbox{and} \quad \| \log r (f_n) - P (r (f_n)) \| \le \varepsilon \; .
\]
Using the estimate $|\tau_{\Nh\Gamma} A| = |(Ae , e)| \le \| A \|$ for any $A \in \Nh\Gamma$, we obtain:
\[
|\tau_{\Nh\Gamma} (\log r(f)) - \tau_{\Nh\Gamma_n} (\log r (f_n))| \le 2 \varepsilon + |\tau_{\Nh\Gamma} (P (f)) - \tau_{\Nh\Gamma_n} (P (f_n))| \; .
\]
Assertion \eqref{eq:6} now implies formula \eqref{eq:3b} in theorem \ref{t17}. For the proof of \eqref{eq:3a} we can assume as above that $r (f)$ and the $r (f_n)$ are positive operators. The relations \eqref{eq:7} imply that the spectra of $r (f)$ and of all $r (f_n)$ lie in $J = [0 , \| f \|_1]$. Choose a sequence of polynomials $P_k (X) \in \R [X]$. converging pointwise to $\log$ in $J$ and satisfying the inequalities $P_k > P_{k+1} > \log$ in $J$ for all $k$. One may obtain such a sequence $(P_k)$ as follows. The continuous functions $\varphi_k$ on $J$ defined by $\varphi_k (x) = 1/k + \log x$ for $x \ge 1/k$ and by $\varphi_k (x) = 1/k + \log 1/k$ for $0 \le x \le 1/k$ satisfy the inequalities $\varphi_k > \varphi_{k+1} > \log$ in $J$ and converge pointwise to $\log$ in $J$. Setting 
\[
\psi_k = 2^{-1} (\varphi_k + \varphi_{k+1})\quad \mbox{and} \quad \varepsilon_k = \min_{x \in J} (\varphi_k (x) - \varphi_{k+1} (x)) > 0 \; ,
\]
the Weierstra{\ss} approximation theorem provides us with polynomials $P_k$ such that
\[
\max_{x\in J} |\psi_k (x) - P_k (x) | \le \frac{\varepsilon_k}{4} \; .
\]
They satisfiy the estimates $\varphi_k > P_k > \varphi_{k+1}$ for all $k$ and hence have the desired properties. It follows that we have
\begin{equation}
  \label{eq:9}
  \lim_{k\to\infty} \tau_{\Nh\Gamma} (P_k (f)) = \log \ddet_{\Nh\Gamma} f \; .
\end{equation}
To see this, consider the spectral resolution $E_{\lambda}$ of the operator $r (f)$. Then we have by the definition of $\det_{\Nh\Gamma} f$:
\begin{eqnarray*}
  \log \ddet_{\Nh\Gamma} f & = & \lim_{\varepsilon \to 0+} \tau_{\Nh\Gamma} (\log (r (f)+\varepsilon)) = \lim_{\varepsilon \to 0+} \int_J \log (\lambda + \varepsilon) \, d\tau_{\Nh\Gamma} (E_{\lambda}) \\
& \overset{(a)}{=} & \int_J \log \lambda \, d\tau_{\Nh\Gamma} (E_{\lambda}) \\
& \overset{(b)}{=} & \lim_{k\to\infty} \int_J P_k (\lambda) \, d\tau_{\Nh\Gamma} (E_{\lambda}) \\
& = & \lim_{k\to\infty} \tau_{\Nh\Gamma} (P_k (f)) \; .
\end{eqnarray*}
Here equations (a) and (b) hold because of Levi's theorem in integration theory (with respect to the finite measure $d\tau_{\Nh\Gamma} (E_{\lambda})$ on $J$). Noting the estimate
\begin{equation}
  \label{eq:10}
  \tau_{\Nh\Gamma_n} (P_k (f_n)) \ge \tau_{\Nh\Gamma_n} (\log (f_n)) 
\end{equation}
we obtain the relations:
\begin{eqnarray*}
  \log \ddet_{\Nh\Gamma} f & \overset{\eqref{eq:9}}{=} & \lim_{k\to\infty} \tau_{\Nh\Gamma} (P_k (f)) \overset{\eqref{eq:6}}{=} \lim_{k\to\infty} \lim_{n\to\infty} \tau_{\Nh\Gamma_n} (P_k (f_n)) \\
& \overset{\eqref{eq:10}}{\ge} & \varlimsup_{n\to\infty} \tau_{\Nh\Gamma_n} (\log f_n) = \varlimsup_{n\to\infty} \log \ddet_{\Nh\Gamma_n} f_n \; .
\end{eqnarray*}
\end{proofof}

\begin{rem}
  If $f \in L^1 (\Gamma)$ is not invertible the question whether the equality \eqref{eq:3b} still holds becomes much more subtle. In the situation of example \ref{t19}, L\"uck has given a criterion in terms of the asymptotic behaviour near zero of the spectral density function, which is hard to verify however, c.f. \cite{Lue1} theorem 3.4, 3. Note that he discusses a slightly different version of the Fuglede--Kadison determinant where the zero-eigenspace is discarded. If $A \in \Nh\Gamma$ is injective on $L^2 (\Gamma)$ the two versions of the $FK$-determinant agree. Incidentally, for a finitely generated amenable group $\Gamma$, a non-zero divisor $f \in \C\Gamma$ has the property that $r (f)$ is injective on $L^2 (\Gamma)$, see \cite{E}. 
\end{rem}

For $\Gamma = \Z$ and the projections to $\Gamma_n = \Z /n $ the above question is related to the theory of diophantine approximation. This was first noted in ergodic theory because for $f \in \Z [\Z]$ the limit
\[
\lim_{n\to \infty} \log \ddet_{\Nh\Z/n} (f_n) = \lim_{n\to\infty} n^{-1} \log \det (r (f_n)) \quad \mbox{(if it exists)}
\]
is the logarithmic growth rate of the number of periodic points of a toral automorphism with characteristic polynomial $\hat{f} \in \Z [X^{\pm 1}]$. One wanted to know if it is equal to the topological entropy which turns out to be given by $m (\hat{f}) = \log \det_{\Nh\Z} (f)$. Using a theorem of Gelfond this was proved by Lind in \cite{Li} \S\,4. See also \cite{S} Lemma 13.53. On the other hand there are examples of non-invertible $f \in L^1 (\Z)$ with $\hat{f} \in \R [X, X^{-1}]$ a linear polynomial for which formula \eqref{eq:3b} is false, see \cite{Lue2}, Example 13.69.

On the other hand, for the sequence $\varphi_r : \Gamma = \Z^d \to \Gamma_r$ from example \ref{t15} formula \eqref{eq:3b} holds for {\it all} $f \in \C [\Z^d]$ as follows from Lawton's theorem \ref{t14} above. One may interpret his proof as an estimate for the spectral density function of $|f|$ near zero. 

These cases suggest the following problem:

\begin{question}
  \label{t21}
In the situation of theorem \ref{t17} consider $f$ in $\Z\Gamma$. Is it true that we have
\[
\ddet_{\Nh\Gamma} f = \lim_{n\to\infty} \ddet_{\Nh\Gamma_n} f_n \quad \mbox{if} \; K_n \to e
\]
even if $f$ is not invertible in $L^1 (\Gamma)$?
\end{question}

In the rest of this section we extend the previous theory somewhat by replacing the maps $\varphi_n : \Gamma \to \Gamma_n$ by a sequence of ``correspondences''. Thus, we consider discrete groups and homomorphisms
\[
\Gamma \xleftarrow{\varphi} \tilde{\Gamma} \xrightarrow{\varphi_n} \Gamma_n \quad \mbox{with kernels}\; K = \Ker \varphi \; \mbox{and} \; K_n = \Ker \varphi_n \; .
\]
Given $\tilde{f} \in L^1 (\tGamma)$ write $f = \varphi_* (\tf) \in L^1 (\Gamma)$ and $f_n = \varphi_{n*} (\tf) \in L^1 (\Gamma_n)$. We will write $K_n \to K$ if one of the following equivalent conditions holds:

{\bf a} No element $\tgamma \in \tGamma$ is contained in $K \btu K_n$ for infinitely many $n$.\\
{\bf b} For any finite subset $Q \subset \tGamma$ we have $(K \btu K_n) \cap Q = \emptyset$ if $n$ is large enough.

Then theorem \ref{t17} has the following generalization:

\begin{theorem}
  \label{t22}
Consider diagrams of countable groups $\Gamma \xleftarrow{\varphi} \tGamma \xrightarrow{\varphi_n} \Gamma_n$ for $n \ge 1$ as above and fix $\tf \in L^1 (\tGamma)$. Then we have
\begin{equation}
  \label{eq:11}
  \ddet_{\Nh\Gamma} f \ge \varlimsup_{n\to\infty} \ddet_{\Nh\Gamma_n} f_n \quad \mbox{if} \; K_n \to K \; .
\end{equation}
For $\tf \in L^1 (\tGamma)^{\times}$, equality holds:
\begin{equation}
  \label{eq:12}
  \ddet_{\Nh\Gamma} f = \lim_{n\to\infty} \ddet_{\Nh\Gamma_n} f_n \quad \mbox{if} \; K_n \to K \; .
\end{equation}
\end{theorem}

\begin{proof}
  As before one first shows that:
  \begin{equation}
    \label{eq:13}
    \tau_{\Nh\Gamma} (P (f)) = \lim_{n\to\infty} \tau_{\Nh\Gamma_n} (P (f_n)) \quad \mbox{for} \; K_n \to K 
  \end{equation}
whenever $\tf \in L^1 (\tGamma)$ and $P \in \C [X]$. Writing $\tf = \sum_{\tgamma \in \tGamma} a_{\tgamma} \tgamma$ and using the inequality
\[
|\tau_{\Nh\Gamma} (f) - \tau_{\Nh\Gamma_n} (f_n)| \le \textstyle \sum\limits_{\tgamma \in K \btu K_n} |a_{\tgamma}| \; ,
\]
we can argue as in the proof of formula \eqref{eq:6}. 

The rest of the proof is analogous to the one of theorem \ref{t17} if we note that the spectra of $r (f)$ and $r (f_n)$ lie in $[0, \| \tf \|_1]$ for $\tf \in L^1 (\tGamma)$ and in $[\| \tf^{-1} \|^{-1}_1 , \| \tf \|_1]$ if $\tf$ is invertible in $L^1 (\tGamma)$. This follows from the estimates
\[
\| r (f) \| \le \| f \|_1 \le \| \tf \|_1 \quad \mbox{and} \quad \| r (f_n) \|  \le \| f_n \|_1 \le \|\tf \|_1 \quad \mbox{if} \; \tf \in L^1 (\tGamma) 
\]
and similar ones for the inverses of $\tf , f , f_n$ in case $\tf \in L^1 (\Gamma)^{\times}$.
\end{proof}

Next, assume that $\tGamma$ is finitely generated and the maps $\varphi$ and $\varphi_n$ are surjective. A family of generators $\tilde{S}$ of $\tGamma$ gives families of generators $S$ and $S_n$ for $\Gamma$ and $\Gamma_n$. If $d = |\tS|$ one can show as in example \ref{t16} that the condition $K_n \to K$ is equivalent to $[\Gamma_n , S_n] \to [\Gamma, S]$ in $X_d$ for $n\to \infty$. \\
For completeness let us give the argument for the implication needed in the following corollary. Assume that $\tgamma \in \tGamma$ is contained in $K \btu K_n$ for infinitely many $n$. Choose a word $\tWh$ in $\tGamma$ with $\tgamma = \ev (\tWh)$. Via $\varphi_n, \varphi$ we obtain words $\Wh_n$ and $\Wh$ in $\Gamma_n$ resp. $\Gamma$ with $\gamma_n = \ev (\Wh_n)$ and $\gamma = \ev (\Wh)$. By assumption there are infitely many $n$, such that $\Wh$ is a relation in $\Gamma$ but $\Wh_n$ is not a relation in $\Gamma_n$ or vice versa. This is not possible however, since for large $n$ the relations of length $\le l (\tWh)$ in $\Gamma$ and $\Gamma_n$ are in canonical bijection if $[\Gamma_n , S_n] \to [\Gamma , S]$. 

\begin{cor}
  \label{t23}
{\bf a} In the situation above, we have for $\tf \in L^1 (\tGamma)$
\[
\lim_{n\to\infty} \tau_{\Nh\Gamma_n} (f_n) = \tau_{\Nh\Gamma} (f) \quad \mbox{if} \quad  [\Gamma_n , S_n] \to [\Gamma , S] \; \mbox{in} \; X_d \; .
\]
{\bf b} If $\tf$ is invertible in $L^1 (\tGamma)$, we have in addition
\[
\lim_{n\to\infty} \ddet_{\Nh\Gamma_n} (f_n) = \ddet_{\Nh\Gamma} (f) \quad \mbox{if} \quad [\Gamma_n , S_n] \to [\Gamma , S] \; \mbox{in} \; X_d \; .
\]
\end{cor}

\begin{proof}
  The condition $[\Gamma_n , S_n] \to [\Gamma , S]$ implies that $K_n \to K$ and hence {\bf a} follows from equation \eqref{eq:13} and {\bf b} from theorem \ref{t22}, \eqref{eq:12}.
\end{proof}

\begin{cor}
  \label{t24}
Consider the free group $F_d$ on $d$-generators $g_1 , \ldots , g_d$. For $[\Gamma , S]$ in $X_d$ define an epimorphism $\varphi : F_d \to \Gamma$ by setting $\varphi (g_i) = s_i$ if $S = (s_1 , \ldots , s_d)$.\\
{\bf a} For every $\tf \in L^1 (F_d)$, the following function is continuous:
\[
T (\tf) : X_d \to \C \quad \mbox{defined by} \; T (\tf) [\Gamma , S] = \tau_{\Nh\Gamma} (\varphi (\tf)) \; .
\]
{\bf b} For every $\tf \in L^1 (F_d)^{\times}$, the function
\[
D (\tf) : X_d \to \R^{> 0} \quad \mbox{defined by} \; D (\tf) [\Gamma , S] = \ddet_{\Nh\Gamma} (\varphi (\tf)) 
\]
is continuous.
\end{cor}

\begin{rems}
  The map $T(\tf)$ depends only on the image of $\tf$ in the quotient of $L^1 (F_d)$ by the subgroup generated by the commutators $[g,h] = gh - hg$. Moreover $D (\tf)$ depends only on the image of $\tf$ in the abelianization of $L^1 (F_d)^{\times}$. Note that assertion {\bf b} is not a formal consequence of {\bf a} since there is no functional calculus in $L^1 (\Gamma)$ allowing us to define the logarithm on all invertible elements of the form $f^* f$.
\end{rems}

\section{Further problems}
\label{sec:4}

For a non-zero polynomial $P$ in $\Z [X^{\pm 1}_1 , \ldots , X^{\pm 1}_d]$ it is well known that the Mahler measure satisfies the inequality $M (P) \ge 1$. In fact $m (P) = \log M (P)$ can be interpreted as the entropy of a suitable $\Z^d$-action and entropies are non-negative, c.f. \cite{LSW}. For discrete groups $\Gamma$ the question whether $\det_{\Nh\Gamma} f \ge 1$ holds for $f \in \Z\Gamma$ has been much studied for the modified version of $\det_{\Nh\Gamma}$ where the zero eigenspace is discarded, c.f. \cite{Lue2} for an overview. If $r (f)$ is injective on $L^2 (\Gamma)$, these results apply to $\det_{\Nh\Gamma} f$ itself. It is known for example that for such $f$ and all residually amenable groups $\Gamma$ we have $\det_{\Nh\Gamma} f \ge 1$. For Mahler measures the polynomials $P$ with $M (P) = 1$ are known by a theorem of Kronecker in the one-variable case and by a result of Schmidt in general, \cite{S}. For them the above mentioned entropy is zero and this is significant for the dynamics. Apart from $\Gamma = \Z^d$ and finite groups $\Gamma$ nothing seems to be known about the following problem:

\begin{question}
  \label{t25}
Given a countable discrete group $\Gamma$, can one characterize the elements $f \in \Z\Gamma$ with $\det_{\Nh\Gamma} f = 1$?
\end{question}

Even the case, where $\Gamma$ is finitely generated and nilpotent would be interesting with the integral Heisenberg group as a starting point. 

The polynomials $P \in \Z [X^{\pm 1}_1 , \ldots , X^{\pm 1}_d]$ with $M (P) = 1$ are either units in $\Z [X^{\pm 1}_1 , \ldots , X^{\pm 1}_d]$ or they have zeros on $T^d$ and hence are not invertible in $L^1$. For $f \in \Z [\Z^d] \cap L^1 (\Z^d)^{\times}$ we therefore have $\det_{\Nh\Z^d} f > 1$ unless $f$ is a unit in $\Z [\Z^d]$. Is the same true in general?

\begin{question}
  \label{t26}
Given a countable discrete group $\Gamma$ and an element $f \in \Z\Gamma \cap L^1 (\Gamma)^{\times}$ which does not have a left inverse in $\Z\Gamma$, is $\det_{\Nh\Gamma} f > 1$?
\end{question}

\begin{rem}
  If $\Gamma$ is residually finite and amenable, the answer is affimative. This was shown in the proof of \cite{DS} Corollary 6.7 by interpreting $\log \det_{\Nh\Gamma} f$ as an entropy and proving that the latter was positive. Note that if $f$ does have a left inverse in $\Z\Gamma$ i.e. $gf = 1$ for some $g \in \Z\Gamma$ we have $(\det_{\Nh\Gamma} g) (\det_{\Nh\Gamma} f) = 1$ which implies that $\det_{\Nh\Gamma} f = 1 = \det_{\Nh\Gamma} g$ if both determinants are $\ge 1$. Incidentally, by a theorem of Kaplansky, $\Nh\Gamma$ and hence also the subrings $\C\Gamma$ and $L^1 (\Gamma)$ are directly finite, i.e. left units are right units and vice versa.
\end{rem}

The last topic we want to mention concerns a continuity property. Answering a question of Schinzel, Boyd proved the following result about the Mahler measure in \cite{B1}:

\begin{theorem}[Boyd]
  \label{t27}
For any Laurent polynomial $P \in \C [X^{\pm 1}_1 , \ldots , X^{\pm 1}_d]$ the function $z \mapsto M (z-P)$ is continuous in $\C$.
\end{theorem}

The proof is based on an estimate due to Mahler which in turn uses Jensen's formula.

Thus the question arises whether $\det_{\Nh\Gamma} (z-f)$ is a continuous function of $z \in \C$ for $f$ in $\C\Gamma$. For $A$ in $\Nh\Gamma$ the function $\varphi (z) = \log \det_{\Nh\Gamma} (z-A)$ is a subharmonic function on $\C$ c.f. \cite{Br} and in particular it is upper semicontinuous. For $z \notin \sigma (A)$ (or even for $z$ outside the support of the Brown measure) the function $\varphi (z)$ is easily seen to be continuous. If $\Gamma$ is finite then $\det_{\Nh\Gamma} (z-f) = |\det (z-r (f))|^{1/|\Gamma|}$ is clearly continuous for $z \in \C$. For the discrete Heisenberg group $\Gamma$ one may use formula (4) in \cite{D2} to get examples where $\det_{\Nh\Gamma} (z-f)$ can be expressed in terms of ordinary integrals. In all these cases one obtains a continuous function of $z$ if $f$ is in $\C \Gamma$.

\begin{thebibliography}{Boy98b}

\bibitem[BL08]{BL}
David~P. Blecher and Louis~E. Labuschagne.
\newblock Applications of the {F}uglede-{K}adison determinant: {S}zeg\"o's
  theorem and outers for noncommutative {$H\sp p$}.
\newblock {\em Trans. Amer. Math. Soc.}, 360(11):6131--6147, 2008.

\bibitem[Boy98a]{B2}
David~W. Boyd.
\newblock Mahler's measure and special values of {$L$}-functions.
\newblock {\em Experiment. Math.}, 7(1):37--82, 1998.

\bibitem[Boy98b]{B1}
David~W. Boyd.
\newblock Uniform approximation to {M}ahler's measure in several variables.
\newblock {\em Canad. Math. Bull.}, 41(1):125--128, 1998.

\bibitem[Boy02]{B3}
David~W. Boyd.
\newblock Mahler's measure and invariants of hyperbolic manifolds.
\newblock In {\em Number theory for the millennium, {I} ({U}rbana, {IL},
  2000)}, pages 127--143. A K Peters, Natick, MA, 2002.

\bibitem[Bro86]{Br}
L.~G. Brown.
\newblock Lidski\u\i's theorem in the type {${\rm II}$} case.
\newblock In {\em Geometric methods in operator algebras ({K}yoto, 1983)},
  volume 123 of {\em Pitman Res. Notes Math. Ser.}, pages 1--35. Longman Sci.
  Tech., Harlow, 1986.

\bibitem[CG]{CG}
Christophe Champetier and Vincent Guirardel.
\newblock Limit groups as limits of free groups: compactifying the set of free
  groups.
\newblock arXiv:math/0401042.

\bibitem[Den]{D2}
C.~Deninger.
\newblock Determinants on von {N}eumann algebras, {M}ahler measures and
  {L}japunov exponents.
\newblock arXiv:0712.0667.

\bibitem[Den06]{D1}
Christopher Deninger.
\newblock Fuglede-{K}adison determinants and entropy for actions of discrete
  amenable groups.
\newblock {\em J. Amer. Math. Soc.}, 19(3):737--758 (electronic), 2006.

\bibitem[DS07]{DS}
Christopher Deninger and Klaus Schmidt.
\newblock Expansive algebraic actions of discrete residually finite amenable
  groups and their entropy.
\newblock {\em Ergodic Theory Dynam. Systems}, 27(3):769--786, 2007.

\bibitem[Ele03]{E}
G{\'a}bor Elek.
\newblock On the analytic zero divisor conjecture of {L}innell.
\newblock {\em Bull. London Math. Soc.}, 35(2):236--238, 2003.

\bibitem[FK52]{FK}
Bent Fuglede and Richard~V. Kadison.
\newblock Determinant theory in finite factors.
\newblock {\em Ann. of Math. (2)}, 55:520--530, 1952.

\bibitem[HS]{HS}
U.~Haagerup and H.~Schultz.
\newblock Invariant subspaces for operators in a general {II}$_1$-factor.
\newblock ArXiv:math.OA/0611256.

\bibitem[Lal08]{L}
Matilde~N. Lal{\'{\i}}n.
\newblock Mahler measures and computations with regulators.
\newblock {\em J. Number Theory}, 128(5):1231--1271, 2008.

\bibitem[Law83]{La}
Wayne~M. Lawton.
\newblock A problem of {B}oyd concerning geometric means of polynomials.
\newblock {\em J. Number Theory}, 16(3):356--362, 1983.

\bibitem[Lin84]{Li}
D.~A. Lind.
\newblock The entropies of topological {M}arkov shifts and a related class of
  algebraic integers.
\newblock {\em Ergodic Theory Dynam. Systems}, 4(2):283--300, 1984.

\bibitem[Lin05]{Li2}
Douglas Lind.
\newblock Lehmer's problem for compact abelian groups.
\newblock {\em Proc. Amer. Math. Soc.}, 133(5):1411--1416 (electronic), 2005.

\bibitem[Los82]{Lo}
V.~Losert.
\newblock A characterization of groups with the one-sided {W}iener property.
\newblock {\em J. Reine Angew. Math.}, 331:47--57, 1982.

\bibitem[LP79]{LP}
Horst Leptin and Detlev Poguntke.
\newblock Symmetry and nonsymmetry for locally compact groups.
\newblock {\em J. Funct. Anal.}, 33(2):119--134, 1979.

\bibitem[LSW90]{LSW}
Douglas Lind, Klaus Schmidt, and Tom Ward.
\newblock Mahler measure and entropy for commuting automorphisms of compact
  groups.
\newblock {\em Invent. Math.}, 101(3):593--629, 1990.

\bibitem[L{\"u}c94]{Lue1}
W.~L{\"u}ck.
\newblock Approximating {$L\sp 2$}-invariants by their finite-dimensional
  analogues.
\newblock {\em Geom. Funct. Anal.}, 4(4):455--481, 1994.

\bibitem[L{\"u}c02]{Lue2}
Wolfgang L{\"u}ck.
\newblock {\em {$L\sp 2$}-invariants: theory and applications to geometry and
  {$K$}-theory}, volume~44 of {\em Ergebnisse der Mathematik und ihrer
  Grenzgebiete. 3. Folge. A Series of Modern Surveys in Mathematics [Results in
  Mathematics and Related Areas. 3rd Series. A Series of Modern Surveys in
  Mathematics]}.
\newblock Springer-Verlag, Berlin, 2002.

\bibitem[Sch95]{S}
Klaus Schmidt.
\newblock {\em Dynamical systems of algebraic origin}, volume 128 of {\em
  Progress in Mathematics}.
\newblock Birkh\"auser Verlag, Basel, 1995.

\bibitem[Sch01]{Sc}
Thomas Schick.
\newblock {$L\sp 2$}-determinant class and approximation of {$L\sp 2$}-{B}etti
  numbers.
\newblock {\em Trans. Amer. Math. Soc.}, 353(8):3247--3265 (electronic), 2001.

\bibitem[Sim05]{Si}
Barry Simon.
\newblock {\em Orthogonal polynomials on the unit circle. {P}art 1}, volume~54
  of {\em American Mathematical Society Colloquium Publications}.
\newblock American Mathematical Society, Providence, RI, 2005.
\newblock Classical theory.

\bibitem[Sze15]{Sz}
G.~Szeg{\"o}.
\newblock Ein {G}renzwertsatz \"uber die {T}oeplitzschen {D}eterminanten einer
  reellen positiven {F}unktion.
\newblock {\em Math. Ann.}, 76(4):490--503, 1915.

\bibitem[Wie32]{Wi}
Norbert Wiener.
\newblock Tauberian theorems.
\newblock {\em Ann. of Math. (2)}, 33(1):1--100, 1932.

\end{thebibliography}

\end{document}